\documentclass[reqno,12pt]{amsart}

\usepackage{amsmath, amssymb}
\usepackage[hypertex]{hyperref}
\usepackage{mathrsfs}
\usepackage{color}

\numberwithin{equation}{section}

\newcommand{\1}{\mathbf {1}}
\newcommand{\0}{\mathbf {0}}
\newcommand{\Z}{{\mathbb Z}}
\newcommand{\Q}{{\mathbb Q}}
\newcommand{\C}{{\mathbb C}}

\newcommand{\CC}{{\mathcal C}}

\newcommand{\bsa}{\boldsymbol{a}}

\newcommand{\al}{\alpha}

\newcommand{\om}{\omega}

\newcommand{\la}{\langle}
\newcommand{\ra}{\rangle}

\newcommand{\mfsl}{\mathfrak{sl}}
\newcommand{\whmfsl}{\widehat{\mathfrak{sl}}}
\newcommand{\hsl}{\whmfsl}

\newcommand{\orbit}{\mathscr{O}}

\DeclareMathOperator{\Aut}{Aut}

\DeclareMathOperator{\Hom}{Hom}
\DeclareMathOperator{\Irr}{Irr}

\DeclareMathOperator{\wt}{wt}

\newcommand{\abs}[1]{\lvert{#1}\rvert}

\newcommand{\SC}[1]{\Irr(#1)_{\mathrm{sc}}}

\newtheorem{theorem}{Theorem}[section]
\newtheorem{proposition}[theorem]{Proposition}
\newtheorem{lemma}[theorem]{Lemma}
\newtheorem{corollary}[theorem]{Corollary}
\newtheorem{remark}[theorem]{Remark}

\begin{document}

\title[$\Z_{2k}$-code VOAs]
{$\Z_{2k}$-code vertex operator algebras}

\author[H. Yamada]{Hiromichi Yamada}
\address{Department of Mathematics, Hitotsubashi University, Kunitachi,
Tokyo 186-8601, Japan}
\email{yamada.h@r.hit-u.ac.jp}

\author[H. Yamauchi]{Hiroshi Yamauchi}
\address{Department of Mathematics, Tokyo Woman's Christian University, 
Suginami, Tokyo 167-8585, Japan}
\email{yamauchi@lab.twcu.ac.jp}



\begin{abstract}
We study a simple, self-dual, rational, and $C_2$-cofinite vertex operator algebra of 
CFT-type whose simple current modules are graded by $\Z_{2k}$. 
Based on those simple current modules, a vertex operator algebra 
associated with a $\Z_{2k}$-code is constructed.  
The classification of irreducible modules for such a vertex operator algebra is established. 
Furthermore, all the irreducible modules are realized 
in a module for a certain lattice vertex operator algebra. 
\end{abstract}

\maketitle

\section{Introduction}\label{sec:intro}

Let $V$ be a simple, self-dual, rational, and $C_2$-cofinite vertex operator algebra 
of CFT-type. Then the set $\SC{V}$ of equivalence classes of simple current 
$V$-modules is closed under the fusion product $\boxtimes_V$, 
and it is graded by a finite abelian group, say $C$. 
That is, $\SC{V} = \{ A^\alpha \mid \alpha \in C \}$ with $A^\alpha$, $\alpha \in C$, 
being inequivalent to each other, $A^0 = V$, 
and $A^\alpha \boxtimes_V A^\beta = A^{\alpha+\beta}$ for $\alpha$, $\beta \in C$. 
If $D$ is a subgroup of $C$ such that the conformal weight $h(A^\alpha)$ of 
$A^\alpha$ is an integer for $\alpha \in D$, then 
the direct sum $\bigoplus_{\alpha \in D} A^\alpha$ has either
a simple vertex operator algebra structure 
or a simple vertex operator superalgebra structure, 
which extends the $V$-module structure \cite[Theorem 3.12]{CKL2015}.   
In such a case $\bigoplus_{\alpha \in D} A^\alpha$ is called 
a simple current extension of $V$. 

The theory of simple current extensions of vertex operator algebras 
has been developed extensively, see for example 
\cite{Carnahan2014, CKL2015, CKM2017, DLM1996a, vEMS2017, Hoehn2003, HKL2015, KO2002}.  
Nowadays it is not hard to construct new vertex operator (super)algebras 
as simple current extensions of known ones. 
One of the examples is the $\Z_k$-code vertex operator algebra \cite{AYY2019}, 
which is a $D$-graded simple current extension 
of the tensor product $K(\mathfrak{sl}_2,k)^{\otimes \ell}$ of $\ell$ copies 
of the parafermion vertex operator algebra $K(\mathfrak{sl}_2,k)$ associated with 
$\mathfrak{sl}_2$ and an integer $k \ge 2$, and $D$ is an additive subgroup of 
$(\Z_k)^\ell$ satisfying a certain condition. 
 
In this paper, we study a $D$-graded simple current extension $U_D$ 
of the tensor product $(U^0)^{\otimes \ell}$ of $\ell$ copies of a vertex operator algebra 
$U^0$ such that $\SC{U^0}$ is graded by $\Z_{2k}$ for an integer $k \ge 2$. 
Here $D$ is an additive subgroup of $(\Z_{2k})^\ell$. 

Let $N = \sqrt{2}A_{k-1}$ be $\sqrt{2}$ times an $A_{k-1}$ root lattice. 
Then a vertex operator algebra $V_N$ associated with the lattice $N$ contains 
a subalgebra isomorphic to 
\[
   L(c_1,0) \otimes \cdots \otimes L(c_{k-1}, 0) \otimes K(\mathfrak{sl}_2,k),
\]
where $L(c_m,0)$, $1 \le m \le k-1$, are the Virasoro vertex operator algebras of 
discrete series.  
The vertex operator algebra $U^0$ is defined to be the commutant of 
$S = L(c_1,0) \otimes \cdots \otimes L(c_{k-2}, 0)$ in $V_N$. 
The vertex operator algebra $U^0$ was previously known 
\cite{Adamovic2007, CKM2017}. 
There are two descriptions of $U^0$, 
one is a $\Z_{k-1}$-graded simple current extension of 
$K(\mathfrak{sl}_2,k-1) \otimes V_{\Z d}$ with $\langle d,d \rangle = 2(k-1)k$, 
and the other is 
a non-simple current extension of $L(c_{k-1},0) \otimes K(\mathfrak{sl}_2,k)$ 
(Theorem \ref{thm:U0}). 
We review the irreducible $U^0$-modules (Theorem \ref{thm:irr_U0-mod}) 
and fusion rules (Theorem \ref{thm:fusion_prod_U0}) in our notation. 
Moreover, we show how the irreducible $U^0$-modules appear 
in $V_{N^\circ}$ (Lemma \ref{lem:Uil_in_VNia}). 

The $\Z_{2k}$-grading of the set $\SC{U^0} = \{ U^l \mid 0 \le l < 2k \}$ 
of equivalence classes of simple current $U^0$-modules 
corresponds to the $\Z_{2k}$-part of the discriminant group 
$N^\circ/N \cong (\Z_2)^{k-2} \times \Z_{2k}$ of the lattice $N$ (Eq. \eqref{eq:Ul_in_Nl}). 
We also recall that the $\Z_k$-grading of $\SC{K(\mathfrak{sl}_2,k)}$ 
used in \cite{AYY2019} for $\Z_k$-code vertex operator algebras 
corresponds to the $\Z_k$-part of $(\Z_2)^{k-2} \times \Z_{2k}$. 

Once the necessary properties of the vertex operator algebra $U^0$ are obtained, 
the construction of the vertex operator algebra $U_D$ is straightforward. 
In fact, $U_D$ is defined to be the commutant of $S^{\otimes \ell}$ 
in a vertex operator algebra $V_{\Gamma_D}$ associated with a certain positive definite 
integral lattice $\Gamma_D$ (Theorem \ref{thm:U_D}). 
It turns out that $U_D$ is a direct sum of a $D$-graded set of simple current 
$(U^0)^{\otimes \ell}$-modules 
$U_\xi = U^{\xi_1} \otimes \cdots \otimes U^{\xi_\ell}$, $\xi = (\xi_1, \ldots, \xi_\ell) \in D$. 
We construct all the irreducible $\chi$-twisted $U_D$-modules for 
$\chi \in \Hom(D, \C^\times)$ in $V_{(N^\circ)^\ell}$, and classify them 
(Theorems \ref{thm:contain_irred} and \ref{thm:irred_U_D-mod}). 
The arguments concerning $U_D$ and its irreducible $\chi$-twisted modules 
are similar to those in Sections 8 and 9 of \cite{AYY2019}. 

If $k = 2$, then $U^0$ is isomorphic to a rank one lattice vertex operator algebra $V_{2\Z}$. 
In the case $k = 4$, the $\Z_8$-code vertex operator algebra $U_D$ 
was studied in \cite{KLY2001}. 
Our result is a generalization of \cite{KLY2001} to an arbitrary $k \ge 2$. 

This paper is organized as follows. 
In Section \ref{sec:preliminaries}, we recall basic properties of 
the Virasoro vertex operator algebra of discrete series and 
the parafermion vertex operator algebra associated with 
$\mathfrak{sl}_2$ and a positive integer $k$.
In Section \ref{sec:VOA_U0}, we define the vertex operator algebra $U^0$ 
and describe it in two ways. We classify irreducible $U^0$-modules as well. 
Fusion rules for irreducible $U^0$-modules are discussed in Section \ref{sec:fusion_rule_U0}.
In Section \ref{sec:UD}, we introduce the positive definite integral lattice 
$\Gamma_D$ and the vertex operator algebra or a vertex operator superalgebra 
$U_D$ for a $\Z_{2k}$-code $D$. 
Finally, in Section \ref{sec:rep_of_U_D}, we classify the irreducible $U_D$-modules.
The weight and the dimension of the top level of the simple current 
$U^0$-module $U^l$, $0 \le l < 2k$, are calculated in Appendix A. 

We use the symbol $\boxtimes_V$ to denote the fusion product over 
a vertex operator algebra $V$.  
We also use the symbol $\otimes$ to denote the tensor product of vertex operator algebras 
and their modules as in \cite{FHL1993}.

\noindent
\textbf{Acknowledgments.}
We would like to thank Ching Hung Lam for valuable discussions.  
Part of this work was done while the first author was staying at 
Institute of Mathematics, Academia Sinica, Taiwan as a visiting scholar from 
May 1, 2017 through April 30, 2018. 
The second author was partially supported by JSPS KAKENHI grant No.19K03409.

\section{Preliminaries}\label{sec:preliminaries}

In this section, we recall the vertex operator algebra associated with 
the discrete series of Virasoro algebra and 
the parafermion vertex operator algebra associated with 
$\mathfrak{sl}_2$ and a positive integer $k$.

\subsection{Virasoro vertex operator algebra $L(c_m,0)$}\label{subsec:Virasoro_VOA}

Let $L(c,0)$ be the simple Virasoro vertex operator algebra of central charge $c$, 
and let $L(c,h)$ be its irreducible highest weight module with highest weight $h$. 
Let 
\begin{equation*}
  c_m = 1 - \frac{6}{(m+2)(m+3)}
\end{equation*}
for $m = 1,2, \ldots$, and
\begin{equation*}
  h^{(m)}_{r,s} = \frac{\big( r(m+3) - s(m+2) \big)^2 - 1}{4(m+2)(m+3)}
\end{equation*}
for $1 \le r \le m+1$, $1 \le s \le m+2$. 
The vertex operator algebra 
$L(c_m,0)$ is self-dual, rational, $C_2$-cofinite, and of CFT-type with 
$L(c_m, h^{(m)}_{r, s})$, $1 \le s \le r \le m+1$,  
a complete set of representatives 
of equivalence classes of irreducible $L(c_m,0)$-modules \cite[Theorem 4.2]{Wang1993}. 
The fusion product of irreducible $L(c_m,0)$-modules is as follows 
\cite[Theorem 4.3]{Wang1993}.
\begin{equation}\label{eq:Vir_fusion}
  L(c_m, h^{(m)}_{r_1, s_1}) \boxtimes_{L(c_m,0)} L(c_m, h^{(m)}_{r_2, s_2}) 
  = \sum_{i, j} L(c_m, h^{(m)}_{|r_1-r_2|+2i-1, |s_1-s_2|+2j-1}),
\end{equation}
where the summation is taken over the integers $i$ and $j$ satisfying
\begin{gather*}
  1 \le i \le \min\{r_1, r_2, m+2-r_1, m+2-r_2\},\\
  1 \le j \le \min\{s_1, s_2, m+3-s_1, m+3-s_2\}.
\end{gather*}

\subsection{Parafermion vertex operator algebra $K(\mathfrak{sl}_2,k)$}
\label{subsec:paraf_VOA}

We fix the notation for the parafermion vertex operator algebra $K(\mathfrak{sl}_2,k)$ 
associated with $\mathfrak{sl}_2$ and a positive integer $k$. 
Details about $K(\mathfrak{sl}_2,k)$ can be found in 
\cite{ALY2014, ALY2019, DLWY2010, DLY2009, DW2016}.  
If $k=1$, then $K(\mathfrak{sl}_2,k)$ reduces to  
the trivial vertex operator algebra $\C\1$. 
So we assume that $k \ge 2$. 
Let 
\begin{equation}\label{eq:lattice_L}
  L^{(k)} = \Z\alpha_1 + \cdots + \Z\alpha_k
\end{equation}
with $\la \alpha_i, \alpha_j \ra = 2 \delta_{i, j}$, and set 
$\gamma_k = \alpha_1 + \cdots + \alpha_k$. 

The vertex operator algebra $V_{L^{(k)}}$ contains a subalgebra isomorphic to 
the simple affine vertex operator algebra $L_{\whmfsl_2}(k,0)$ associated with 
the affine Kac-Moody algebra $\whmfsl_2$ at level $k$, 
and $K(\mathfrak{sl}_2,k)$ is realized as 
the commutant of the vertex operator algebra $V_{\Z \gamma_k}$ in $L_{\whmfsl_2}(k,0)$.  
Let 
\begin{equation*}
  M_{(k)}^j 
  = \{ v \in L_{\hsl_2}(k,0) \mid \gamma_k(n) v = -2j \delta_{n,0} v \text{ for } n \ge 0 \}
\end{equation*}
for $0 \le j < k$. 
Then $M_{(k)}^0 = K(\mathfrak{sl}_2,k)$, and
\begin{equation}\label{eq:Lsl2k_dec}
  L_{\whmfsl_2}(k,0) 
  = \bigoplus_{j = 0}^{k-1} M_{(k)}^j \otimes V_{\Z \gamma_k - j \gamma_k/k}.
\end{equation}

An irreducible $L_{\whmfsl_2}(k,0)$-module $L_{\whmfsl_2}(k,i)$ with $i+1$ dimensional top level 
can be constructed in the $V_{L^{(k)}}$-module $V_{(L^{(k)})^\circ}$ for $0 \le i \le k$, 
where $(L^{(k)})^\circ = \frac{1}{2}L^{(k)}$ is the dual lattice of $L^{(k)}$. 
Let
\begin{equation*}
  M_{(k)}^{i,j} 
  = \{ v \in L_{\whmfsl_2}(k,i) \mid \gamma_k(n) v = (i-2j) \delta_{n,0} v \text{ for } n \ge 0 \}
\end{equation*}
for $0 \le j < k$. Then $M_{(k)}^{0,j} = M_{(k)}^j$, and 
\begin{equation}\label{eq:Lsl2ki_dec}
  L_{\whmfsl_2}(k,i) 
  = \bigoplus_{j = 0}^{k-1} M_{(k)}^{i,j} \otimes V_{\Z \gamma_k + (i-2j) \gamma_k/2k}. 
\end{equation}

The index $j$ of $M_{(k)}^j$ and $M_{(k)}^{i,j}$ can be considered to be modulo $k$. 
We will use the following properties of $M_{(k)}^0$ and $M_{(k)}^{i,j}$. 

(1) $M_{(k)}^0 = K(\mathfrak{sl}_2,k)$ is a simple, self-dual, rational, and
$C_2$-cofinite vertex operator algebra of CFT-type with central charge 
$2(k-1)/(k+2)$.

(2) $M_{(k)}^{i,j}$, $0 \le i \le k$, $0 \le j < k$, are irreducible $M_{(k)}^0$-modules with
\begin{equation}\label{eq:isom_Mij}
  M_{(k)}^{i,j} \cong M_{(k)}^{k-i, j-i}, 
\end{equation}
and $M_{(k)}^{i,j}$, $0 \le j < i \le k$, form a complete set of 
representatives of equivalence classes of irreducible $M_{(k)}^0$-modules.

(3) The top level of $M_{(k)}^{i,j}$ is one dimensional and its conformal weight is 
\begin{equation}\label{eq:top-wt-Mkij}
  h(M_{(k)}^{i,j}) = \frac{1}{2k(k+2)}\Big( k(i-2j) - (i-2j)^2 + 2k(i-j+1)j \Big)
\end{equation}
for $0 \le j \le i \le k$. 
Eq. \eqref{eq:top-wt-Mkij} is valid even in the case $j = i$. 

(4) The fusion product of irreducible $M_{(k)}^0$-modules is
\begin{equation}\label{eq:paraf_fusion}
  M_{(k)}^{i_1,j_1} \boxtimes_{M_{(k)}^0} M_{(k)}^{i_2,j_2} 
  = \sum_{r \in R(i_1,i_2)} M_{(k)}^{r, (2j_1 - i_1 + 2j_2 - i_2 + r)/2}, 
\end{equation}
where $R(i_1,i_2)$ is the set of integers $r$ satisfying
\begin{equation*}
  |i_1-i_2| \le r \le \min \{i_1+i_2, 2k - i_1 - i_2\}, \quad i_1+i_2+r \in 2\Z.
\end{equation*} 
In particular, $M_{(k)}^{j}$, $0 \le j < k$, are the simple currents, and
\begin{equation*}
  M_{(k)}^{p} \boxtimes_{M_{(k)}^0} M_{(k)}^{i,j} = M_{(k)}^{i,j+p}.
\end{equation*}

(5) If $k \ge 3$, then the automorphism group $\Aut M_{(k)}^0$ of $M_{(k)}^0$ is 
generated by an involution $\theta$, 
and $M_{(k)}^{i,j} \circ \theta \cong M_{(k)}^{i,i-j}$. 
The automorphism $\theta$ is induced from the $-1$-isometry $\alpha \mapsto -\alpha$ 
of the lattice $L^{(k)}$.

\section{Vertex operator algebra $U^0$}\label{sec:VOA_U0}

In this section, we discuss a vertex operator algebra $U^0$. 
The vertex operator algebra $U^0$ and its irreducible modules have already been studied    
\cite{Adamovic2007}, \cite[Section 4.4.2]{CKM2017}, see also \cite[Section 3.1]{YY2018}. 
We describe $U^0$ and all the irreducible $U^0$-modules 
in a lattice vertex operator algebra and in its module for later use. 

We fix an integer $k \ge 2$. 
Let $\alpha_i$, $1 \le i \le k$, and $L^{(k)}$ be as in Section \ref{subsec:paraf_VOA}. 
Thus $\la \alpha_i, \alpha_j \ra = 2 \delta_{i, j}$ and 
$L^{(k)} = \Z\alpha_1 + \cdots + \Z\alpha_k$. 
Let 
\begin{equation}\label{eq:gammas}
  \gamma_{k-1} = \al_1 + \cdots + \al_{k-1}, 
  \quad \gamma_k = \al_1 + \cdots + \al_k, 
  \quad d = \gamma_{k-1} - (k-1)\al_k.
\end{equation}
Then 
$d = \gamma_k - k\al_k$, 
$\la \gamma_k, \gamma_k \ra = 2k$, 
$\la d, d \ra = 2(k-1)k$, 
and $\la \gamma_k,d \ra = 0$.

For $\bsa = \bsa(k) = (a_1, \ldots,a_k) \in \{ 0,1 \}^k$, 
let $\delta_{\bsa(k)} = \frac{1}{2} \sum_{p=1}^k a_p \alpha_p$. 
The vertex operator algebra $V_{L^{(k)}}$ associated with the lattice $L^{(k)}$ 
contains a subalgebra isomorphic to
\[
  L(c_1,0) \otimes \cdots \otimes L(c_{k-1}, 0) \otimes L_{\whmfsl_2}(k,0), 
\]
and
\begin{equation}\label{eq:V_Lk_delta-dec}
  V_{L^{(k)} +  \delta_{\bsa(k)}} = 
  \bigoplus_{\substack{0 \le i_s \le s\\i_s \equiv b_s \pmod{2}\\1 \le s \le k}}
  L(c_1, h^{(1)}_{i_1+1, i_2+1}) \otimes \cdots \otimes L(c_{k-1},h^{(k-1)}_{i_{k-1}+1, i_{k}+1}) 
  \otimes L_{\widehat{\mfsl}_2}(k,i_{k})
\end{equation}
as an $L(c_1,0) \otimes \cdots \otimes L(c_{k-1}, 0) \otimes L_{\widehat{\mfsl}_2}(k,0)$-module, 
where $b_s = \sum_{p=1}^s a_p$ 
\cite{KR2013, LLY2003, Wakimoto2001}, 
see also \cite[Section 5]{AYY2019}.
Since $L^{(k-1)} \oplus \Z\al_{k} = L^{(k)}$,  
it follows that
\begin{equation}\label{eq:dec_V_Lk_coset_1}
  V_{L^{(k-1)} +  \delta_{\bsa(k-1)}} \otimes V_{\Z \alpha_k + a_k \alpha_k/2} 
  = V_{L^{(k)} +  \delta_{\bsa(k)}}.
\end{equation}

Let $\omega^s$ be the conformal vector of the Virasoro vertex operator algebra 
$L(c_s,0)$, $1 \le s \le k-1$. 
We apply \eqref{eq:V_Lk_delta-dec} to $V_{L^{(k-1)} +  \delta_{\bsa(k-1)}}$ 
for $k-1$ in place of $k$. 
Then
\begin{equation*}
\begin{split}
  &\{ v \in V_{L^{(k-1)} +  \delta_{\bsa(k-1)}} \otimes V_{\Z \alpha_k + a_k \alpha_k/2} \mid 
  \omega_{(1)}^s v = h^{(s)}_{i_{s}+1, i_{s+1}+1} v, 1 \le s \le k-2 \}\\
  & \qquad  = L_{\widehat{\mfsl}_2}(k-1,i_{k-1}) \otimes V_{\Z \alpha_k + a_k \alpha_k/2}.
\end{split}
\end{equation*}

We also have
\begin{equation}\label{eq:VLa_dec-1}
\begin{split}
  &\{ v \in V_{L^{(k)} +  \delta_{\bsa(k)}} \mid  
  \omega_{(1)}^s v = h^{(s)}_{i_{s}+1, i_{s+1}+1} v, 1 \le s \le k-2 \}\\
  & \qquad  = \bigoplus_{\substack{0 \le i_k \le k\\i_k \equiv b_k \pmod{2}}}
  L(c_{k-1},h^{(k-1)}_{i_{k-1}+1, i_{k}+1}) \otimes L_{\widehat{\mfsl}_2}(k,i_{k})
\end{split}
\end{equation}
by \eqref{eq:V_Lk_delta-dec}. 
Hence \eqref{eq:dec_V_Lk_coset_1} implies that
\begin{equation}\label{eq:mult-eq}
  L_{\widehat{\mfsl}_2}(k-1,i_{k-1}) \otimes V_{\Z \alpha_k + a_k \alpha_k/2} 
  = \bigoplus_{\substack{0 \le i_k \le k\\i_k \equiv b_k \pmod{2}}}
  L(c_{k-1},h^{(k-1)}_{i_{k-1}+1, i_{k}+1}) \otimes L_{\widehat{\mfsl}_2}(k,i_{k})
\end{equation}
for $0 \le i_{k-1} \le k-1$ with $i_{k-1} \equiv b_{k-1} \pmod{2}$. 

We set $i = i_{k-1}$ for simplicity of notation.   
Since $\la \gamma_{k-1}, \al_k \ra = 0$, 
the left hand side of \eqref{eq:mult-eq} is 
\begin{equation}\label{eq-2}
  \bigoplus_{j=0}^{k-2} 
  M_{(k-1)}^{i,j} \otimes 
  V_{\Z \gamma_{k-1} + (i-2j) \gamma_{k-1}/2(k-1) + \Z \al_k + a_k \al_k/2}
\end{equation}
by \eqref{eq:Lsl2ki_dec} for $k-1$ in place of $k$.
Since $\gamma_{k-1} = \gamma_k - \al_k$, 
and since $\alpha_k = (\gamma_k - d)/k$,  
we have a coset decomposition
\begin{equation}\label{eq:ZcZd_coset_dec_ZaZb}
  \Z \gamma_{k-1} + \Z \al_k 
  = \bigcup_{p = 0}^{k-1} 
  \Big(  \Z d + \Z \gamma_k + \frac{p}{k}(\gamma_k - d) \Big)
\end{equation}
with $[ \Z \gamma_{k-1} + \Z \al_k : \Z d + \Z \gamma_k ] = k$, and
\begin{equation}\label{coset_gen-1}
\begin{split}
  & \Z \gamma_{k-1} + \frac{i-2j}{2(k-1)} \gamma_{k-1} + \Z \al_k + \frac{1}{2} a_k \al_k\\
  & \quad = \bigcup_{p = 0}^{k-1}  
  \Big( \big( \Z d + \frac{1}{2(k-1)k} (- (k-1)(a_k + 2p) + i - 2j) d \big)\\
  & \hspace{60mm} 
  + \big(\Z \gamma_k + \frac{1}{2k} (a_k + 2p+ i - 2j) \gamma_k \big) \Big)
\end{split}
\end{equation}
for $0 \le j < k-1$.
The left hand side of \eqref{coset_gen-1} is 
a coset of $\Z \gamma_{k-1} + \Z \al_k$,  
and it is determined by $j$ modulo $k-1$, 
whereas the right hand side is a union of cosets of 
$\Z d + \Z \gamma_k$, and 
$j$ can not be considered to be modulo $k-1$ for these cosets.

For $0 \le p < k$, define $0 \le l < 2k$ by 
\begin{equation}\label{eq:def_l}
  l \equiv i + a_k + 2(p-j) \pmod{2k}.
\end{equation}
Then \eqref{coset_gen-1} can be written as
\begin{equation}\label{coset_gen-2}
\begin{split}
  &\Z \gamma_{k-1} + \frac{i-2j}{2(k-1)} \gamma_{k-1} + \Z \al_k + \frac{1}{2} a_k \al_k\\
  &\qquad  = \bigcup_{\substack{0 \le l < 2k\\ l \equiv i + a_k \pmod{2}}}
  \Big( \big( \Z d - \frac{l}{2k} d + \frac{i - 2j}{2(k-1)} d \big)
  + \big( \Z \gamma_k + \frac{l}{2k} \gamma_k \big) \Big).
\end{split}
\end{equation}
Since $\la d, \gamma_k \ra = 0$, 
it follows from \eqref{eq-2} and \eqref{coset_gen-2} that 
\begin{equation}\label{eq:Lki_Vbvep_dec_2nd_form}
\begin{split}
  &L_{\whmfsl_2}(k-1,i) \otimes V_{\Z \al_k + a_k \al_k/2}\\
  &\quad = \bigoplus_{j = 0}^{k-2} 
  \bigoplus_{\substack{0 \le l < 2k\\ l \equiv i + a_k \pmod{2}}}
  M_{(k-1)}^{i,j}  
  \otimes  V_{\Z d - l d/2k + (i - 2j) d/2(k-1)}
  \otimes V_{\Z \gamma_k + l \gamma_k/2k}.
\end{split}
\end{equation}

Let
\begin{equation*}
  U^{i,l} = \{ v \in L_{\whmfsl_2}(k-1,i) \otimes V_{\Z \al_k + a_k \al_k/2} 
  \mid \gamma_k(n)v = l \delta_{n,0}v \text{ for } n \ge 0\}
\end{equation*}
for $0 \le i \le k-1$, $0 \le l < 2k$, 
which is the multiplicity of $V_{\Z \gamma_k + l \gamma_k/2k}$ 
in \eqref{eq:Lki_Vbvep_dec_2nd_form}. 
Then
\begin{equation*}
  U^{i,l} = 
  \bigoplus_{j=0}^{k-2} M_{(k-1)}^{i,j} \otimes V_{\Z d - l d/2k + (i - 2j) d/2(k-1)}.
\end{equation*}
The index $l$ of $U^{i,l}$ can be considered to be modulo $2k$. 

The right hand side of \eqref{eq:mult-eq} is 
\begin{equation}\label{eq:rhs-mult-eq}
  \bigoplus_{\substack{0 \le i_k \le k\\i_k \equiv b_k \pmod{2}}} \bigoplus_{p = 0}^{k-1} 
  L(c_{k-1},h^{(k-1)}_{i_{k-1}+1, i_{k}+1}) 
  \otimes M_{(k)}^{i_k,p} 
  \otimes V_{\Z \gamma_k + (i_k-2p) \gamma_k/2k}
\end{equation}
by \eqref{eq:Lsl2ki_dec}.  
Recall that $l$ is defined in \eqref{eq:def_l} and that $i = i_{k-1}$ satisfies the condition 
$i \equiv b_{k-1} \pmod{2}$. 
Since $b_k = b_{k-1} + a_k$, we have $i_k - l \in 2\Z$ for $i_k$ such that 
$i_k \equiv b_k \pmod{2}$. 
Then $V_{\Z \gamma_k + (i_k-2p) \gamma_k/2k}$ agrees with 
$V_{\Z \gamma_k + l \gamma_k/2k}$ 
if and only if $p \equiv (i_k - l)/2 \pmod{k}$. 
Hence the multiplicity of $V_{\Z \gamma_k + l \gamma_k/2k}$ in \eqref{eq:rhs-mult-eq} is
\begin{equation}\label{eq:def_Uil-bis}
\begin{split}
  &\{ v \in \bigoplus_{\substack{0 \le i_k \le k\\i_k \equiv b_k \pmod{2}}}
  L(c_{k-1},h^{(k-1)}_{i+1, i_{k}+1}) \otimes L_{\widehat{\mfsl}_2}(k,i_{k})
  \mid \gamma_k(n)v = l \delta_{n,0}v \text{ for } n \ge 0\}\\
  &\qquad = \bigoplus_{\substack{0 \le i_k \le k\\i_k \equiv b_k \pmod{2}}} 
  L(c_{k-1},h^{(k-1)}_{i+1, i_{k}+1}) 
  \otimes M_{(k)}^{i_k,(i_k - l)/2}. 
\end{split}
\end{equation}

Therefore, the following theorem is proved. 

\begin{theorem}\label{thm:Uij}
The multiplicity $U^{i,l}$ of $V_{\Z \gamma_k + l \gamma_k/2k}$ 
on both sides of \eqref{eq:mult-eq} is described in two ways, namely, 
\begin{align*}
  U^{i,l} 
  &= \bigoplus_{j=0}^{k-2} 
  M_{(k-1)}^{i,j} \otimes V_{\Z d - l d/2k + (i - 2j) d/2(k-1)}\\
 &= \bigoplus_{\substack{0 \le i_k \le k\\i_k \equiv b_k \pmod{2}}} 
  L(c_{k-1},h^{(k-1)}_{i+1, i_{k}+1}) 
  \otimes M_{(k)}^{i_k,(i_k - l)/2},  
\end{align*}
for $0 \le i \le k-1$ with $i \equiv b_{k-1} \pmod{2}$, and 
$0 \le l < 2k$ with $l \equiv b_k \pmod{2}$. 
The first one is a direct sum of irreducible $M_{(k-1)}^0 \otimes V_{\Z d}$-modules, and 
the second one is a direct sum of irreducible $L(c_{k-1},0) \otimes M_{(k)}^0$-modules.
\end{theorem}

In the case $\bsa(k) = (0, \ldots,0)$, 
take the commutant of $S = L(c_1,0) \otimes \cdots \otimes L(c_{k-2}, 0)$ 
in $V_{L^{(k-1)}} \otimes V_{\Z\al_{k}} = V_{L^{(k)}}$.
Then we have 
\begin{equation}\label{eq:eq-1}
  L_{\whmfsl_2}(k-1,0) \otimes V_{\Z\al_{k}} 
  = \bigoplus_{j=0}^{\lfloor k/2 \rfloor}
  L(c_{k-1}, h^{(k-1)}_{1, 2j+1}) \otimes L_{\whmfsl_2}(k,2j),
\end{equation}
where $\lfloor k/2 \rfloor$ is the largest integer which does not exceed $k/2$. 

Let $U^0$ be the commutant of $V_{\Z \gamma_k}$ in \eqref{eq:eq-1}. 
Then $U^0 = U^{0,0}$ in the notation of Theorem \ref{thm:Uij}. 
In particular, 
\begin{equation}\label{eq:U0-1}
  U^0 = \bigoplus_{j = 0}^{k-2} M_{(k-1)}^{j} \otimes V_{\Z d - jd/(k-1)}, 
\end{equation}
which is a $\Z_{k-1}$-graded simple current extension of $M_{(k-1)}^0 \otimes V_{\Z d}$. 
Hence the following theorem holds by \cite[Theorem 2.14]{Yamauchi2004}, 
see also \cite{Adamovic2007}, \cite[Section 4.4.2]{CKM2017}. 

\begin{theorem}\label{thm:U0} 
\textup{(1)} 
$U^0$ is a simple, self-dual, rational, and $C_2$-cofinite vertex operator algebra 
of CFT-type with central charge $3(k-1)/(k+1)$. 

\textup{(2)}  
$U^0$ is described in two ways, namely, 
\begin{align*}
  U^0 
  &= \bigoplus_{j=0}^{k-2} M_{(k-1)}^{j} \otimes V_{\Z d - jd/(k-1)}\\
  &= \bigoplus_{j=0}^{\lfloor k/2 \rfloor} 
  L(c_{k-1},h^{(k-1)}_{1, 2j+1}) \otimes M_{(k)}^{2j,j}.
\end{align*}
The first one is a $\Z_{k-1}$-graded simple current extension of $M_{(k-1)}^0 \otimes V_{\Z d}$, 
and the second one is a non-simple current extension of $L(c_{k-1},0) \otimes M_{(k)}^0$.
\end{theorem}

The fusion product of irreducible $M_{(k-1)}^0$-modules $M_{(k-1)}^j$  
agrees with the fusion product of irreducible $V_{\Z d}$-modules $V_{\Z d - jd/(k-1)}$ 
under the correspondence
\[
  M_{(k-1)}^j \ \longleftrightarrow \ V_{\Z d - jd/(k-1)}, \quad 0 \le j < k-1. 
\]
Similarly, the fusion product of irreducible $L(c_{k-1},0)$-modules 
$L(c_{k-1},h^{(k-1)}_{1, 2j+1})$ 
agrees with the fusion product of irreducible $M_{(k)}^0$-modules 
$M_{(k)}^{2j,j}$ 
under the correspondence
\[
  L(c_{k-1},h^{(k-1)}_{1, 2j+1}) \ \longleftrightarrow \ M_{(k)}^{2j,j}, \quad 
  0 \le j \le \lfloor k/2 \rfloor.
\]
In fact, for $0 \le p \le q \le \lfloor k/2 \rfloor$, we have 
\[  
  L(c_{k-1},h^{(k-1)}_{1, 2p+1}) \boxtimes_{L(c_{k-1},0)} L(c_{k-1},h^{(k-1)}_{1, 2q+1}) 
  = \sum_{j=0}^{\min\{ 2p, k-2q \}} L(c_{k-1},h^{(k-1)}_{1, 2(q-p+j)+1})
\]
by \eqref{eq:Vir_fusion}, and 
\[
  M_{(k)}^{2p,p} \boxtimes_{M_{(k)}^0} M_{(k)}^{2q,q} 
  = \sum_{j=0}^{\min\{ 2p, k-2q \}} M_{(k)}^{2(q-p+j), q-p+j}
\]
by \eqref{eq:paraf_fusion}.

Since the left hand side of \eqref{eq:Lki_Vbvep_dec_2nd_form} is an 
$L_{\whmfsl_2}(k-1,0) \otimes V_{\Z \al_k}$-module, 
and since $U^0$ is the commutant of $V_{\Z \gamma_k}$ in 
$L_{\whmfsl_2}(k-1,0) \otimes V_{\Z \al_k}$, 
it follows that $U^{i,l}$ is a $U^0$-module. 

For simplicity of notation, set
\begin{gather*}
  A^j = M_{(k-1)}^j \otimes V_{\Z d - j d/(k-1)},\\
  X(i,j,l) = M_{(k-1)}^{i,j} \otimes V_{\Z d - l d/2k + (i - 2j) d/2(k-1)} 
\end{gather*}
for $0 \le i \le k-1$, $0 \le j < k-1$, and $0 \le l < 2k$.
Then 
$U^0 = \bigoplus_{j=0}^{k-2} A^j$   
and 
$U^{i,l} = \bigoplus_{j=0}^{k-2} X(i,j,l)$
as $A^0$-modules.  
For fixed $i$ and $l$, $X(i,j,l)$, $0 \le j < k-1$, 
are inequivalent irreducible $A^0$-modules. 
Thus the $U^0$-module $U^{i,l}$ is irreducible. 
In fact, the $U^0$-module structure on the direct sum $\bigoplus_{j=0}^{k-2} X(i,j,l)$  
which extends the $A^0$-module structure is unique 
\cite[Proposition 3.8]{SY2003}.
Since 
\begin{equation}\label{eq:fusion_prod_summand}
  A^j \boxtimes_{A^0} X(i,0,l) = X(i,j,l), 
\end{equation}
we have
\begin{equation}\label{eq:Uil_as_U0_module}
  U^{i,l} = U^0 \boxtimes_{A^0} X(i,0,l).
\end{equation}

The following lemma is a consequence of the isomorphism 
$M_{(k-1)}^{i,j} \cong M_{(k-1)}^{k-1-i,j-i}$ of $M_{(k-1)}^0$-modules 
in \eqref{eq:isom_Mij} for $k-1$ in place of $k$. 

\begin{lemma}\label{lem:equiv_M0_V_Zb-module} 
\textup{(1)} 
Let $0 \le i, i' \le k-1$, $0 \le j, j' < k-1$, and $0 \le l, l' < 2k$. 
Then $X(i,j,l) \cong X(i',j',l')$ 
as $A^0$-modules if and only if one of the following conditions holds.

\textup{(i)} 
$i = i'$, $j = j'$, and $l = l'$.

\textup{(ii)} 
$i' = k-1-i$, $j' \equiv j-i \pmod{k-1}$, and $l' \equiv k + l \pmod{2k}$.

\textup{(2)} 
$X(i,j,l)$, $0 \le j < i \le k-1$, $0 \le l < 2k$, are inequivalent to each other.
\end{lemma}

The above  lemma implies the next lemma. 

\begin{lemma}\label{lem:equiv_Uiell}
Let $0 \le i, i' \le k-1$ and $0 \le l, l' < 2k$. 
Then $U^{i,l} \cong U^{i',l'}$ as $U^0$-modules 
if and only if one of the following conditions holds.

\textup{(i)}
$i = i'$ and $l = l'$.

\textup{(ii)} 
$i' = k-1-i$ and $l' \equiv k + l \pmod{2k}$. 
\end{lemma}

Next, we calculate the difference of the conformal weight of two irreducible $A^0$-modules. 

\begin{lemma}\label{lem:dif_of_conf_wt}
Let $0 \le i \le k-1$, $0 \le j < k-1$, and $0 \le s < 2(k-1)k$. 
Then the difference of the conformal weight of 
$M_{(k-1)}^{i,p} \otimes V_{\Z d + s d/2(k-1)k - p d/(k-1)}$ for $p = 0,j$ is
as follows.
\begin{equation}\label{eq:diff_wt_MijVZdsj} 
\begin{split}
  & h(M_{(k-1)}^{i,j} \otimes V_{\Z d + s d/2(k-1)k - j d/(k-1)})
  - h(M_{(k-1)}^{i,0} \otimes V_{\Z d + s d/2(k-1)k}) \\
  &\qquad 
  \equiv \frac{j(i - s)}{k-1} \pmod{\Z}.
\end{split}
\end{equation}
\end{lemma}

Indeed, we have $h(M_{(k-1)}^{i,j})$ for $0 \le j \le i \le k-1$ 
by \eqref{eq:top-wt-Mkij} for $k-1$ in place of $k$,  
and $h(M_{(k-1)}^{i,j})$ for $0 \le i \le j \le k-1$ 
is obtained by using \eqref{eq:isom_Mij} for $k-1$ in place of $k$. 
In fact,  
\[
  h(M_{(k-1)}^{i,j}) - h(M_{(k-1)}^{i,0}) \equiv \frac{j(i-j)}{k-1} \pmod{\Z}.
\]
Since 
\[
  h(V_{\Z d + s d/2(k-1)k - j d/(k-1)}) - h(V_{\Z d + s d/2(k-1)k})
  \equiv \frac{j(j-s)}{k-1} \pmod{\Z}, 
\]
Eq. \eqref{eq:diff_wt_MijVZdsj} holds.

The following theorem holds,  
see \cite{Adamovic2007}, \cite[Section 4.4.2]{CKM2017}, \cite[Section 3.1]{YY2018}. 

\begin{theorem}\label{thm:irr_U0-mod}
\textup{(1)} 
Any irreducible $U^0$-module is isomorphic to $U^{i,l}$ for some 
$0 \le i \le k-1$, $0 \le l < 2k$.

\textup{(2)}  
The irreducible $U^0$-modules 
$U^{i,l}$, $0 \le i \le k-1$, $0 \le l < 2k$, 
are inequivalent to each other except for the isomorphism
\begin{equation}\label{eq:Uil_equivalence}
  U^{i,l} \cong U^{k-1-i, k+l}.
\end{equation}

\textup{(3)} 
There are exactly $k^2$ inequivalent irreducible $U^0$-modules. 

\textup{(4)} 
The conformal weight of any irreducible $U^0$-module except for $U^0$ is positive.
\end{theorem}

\begin{proof}
Let $W$ be an irreducible $U^0$-module. 
Then $W$ is a direct sum of irreducible $A^0$-modules. 
Let $X$ be an irreducible $A^0$-submodule of $W$. 
We may assume that 
$X = M_{(k-1)}^{i,0} \otimes V_{\Z d + s d/2(k-1)k}$ 
for some $0 \le i \le k-1$ and $0 \le s < 2(k-1)k$ by \eqref{eq:fusion_prod_summand}. 
Then 
\[
  A^1 \boxtimes_{A^0} X = M_{(k-1)}^{i,1} \otimes V_{\Z d + s d/2(k-1)k - d/(k-1)}
\]
is contained in $W$,  
so $h(A^1 \boxtimes_{A^0}  X) - h(X)$ is an integer. 
Thus $s \equiv i \pmod{k-1}$ by Lemma \ref{lem:dif_of_conf_wt}. 
Then $s \equiv i - m (k-1) \pmod{2(k-1)k}$ for some $0 \le m < 2k$. 
Take $0 \le l < 2k$ such that $l \equiv i + m \pmod{2k}$. 
Then $s \equiv ik - l(k-1) \pmod{2(k-1)k}$,   
and $W = U^{i,l}$ by \eqref{eq:Uil_as_U0_module}. 
Thus the assertion (1) holds. 

Lemma \ref{lem:equiv_Uiell} implies the assertion (2). 
The assertion (3) is clear from (1) and (2). 
The conformal weight of any irreducible $M_{(k-1)}^{0}$-module 
except for $M_{(k-1)}^{0}$ is positive.  
Thus the assertion (4) holds. 
\end{proof}

There are $(k-1)^2k^2$ inequivalent irreducible $A^0$-modules, 
namely, 
\[
  \Irr(A^0) = \{ M_{(k-1)}^{i,j} \otimes V_{\Z d + s d/2(k-1)k} 
  \mid 0 \le j < i \le k-1, 0 \le  s < 2(k-1)k \}. 
\]
Only $(k-1)k^2$ of them can be direct summands 
of an irreducible $U^0$-module.   
In fact, let 
\begin{equation*}
  \Irr^0(A^0) = \{ X(i,j,l) \mid 0 \le j < i \le k-1, 0 \le l <2k \}.
\end{equation*}
Then each irreducible $U^0$-module is a direct sum of 
$k-1$ inequivalent irreducible $A^0$-modules in $\Irr^0(A^0)$.

Recall the automorphism $\theta$ of $M_{(k)}^0$, 
which is induced from the $-1$-isometry $\alpha \mapsto -\alpha$ of 
the lattice $L^{(k)}$. 
The isometry induces an automorphism of $U^0$ of order two as well. 
We denote the automorphism by the same symbol $\theta$. 
Then
\begin{equation}\label{eq:theta_act_Uil}
  U^{i,l} \circ \theta \cong U^{i,-l}.
\end{equation}

\section{Fusion rule of irreducible $U^0$-modules}\label{sec:fusion_rule_U0}

The fusion rule of the irreducible $U^0$-modules 
was previously known \cite[Section 4.4.2]{CKM2017}, 
see also \cite[Section 3.2]{YY2018}. 
The fusion rule in our notation  
for the irreducible $U^0$-modules is as follows. 

\begin{theorem}\label{thm:fusion_prod_U0} 
Let $0 \le i_1, i_2 \le k-1$ and $0 \le l_1, l_2 < 2k$. Then 
\begin{equation}\label{eq:fusion_prod_U0}
  U^{i_1,l_1} \boxtimes_{U^0} U^{i_2,l_2}  =  \sum_{r \in R(i_1,i_2)} U^{r,l_1 + l_2},
\end{equation}
where
$R(i_1,i_2)$ is the set of integers $r$ satisfying 
\begin{equation*}
  |i_1-i_2| \le r \le \min \{i_1+i_2, 2(k-1) - i_1 - i_2\}, \quad i_1+i_2+r \in 2\Z, 
\end{equation*}
and $l_1 + l_2$ is considered to be modulo $2k$. 
The irreducible $U^0$-modules $U^{r,l_1+l_2}$, $r \in R(i_1,i_2)$, on 
the right hand side of \eqref{eq:fusion_prod_U0} are inequivalent to each other. 
\end{theorem}

\begin{proof} 
Let $\CC_{A^0}$ be the category of $A^0$-modules, and  
let $\CC_{A^0}^0$ be the full subcategory of $\CC_{A^0}$ 
consisting of the objects $X$ of $\CC_{A^0}$ such that 
$U^0 \boxtimes_{A^0} X$ is a $U^0$-module \cite[Definition 2.66]{CKM2017}. 
Then $\Irr^0(A^0)$ constitutes the simple objects of $\CC_{A^0}^0$   
\cite[Proposition 2.65]{CKM2017}. 
The category $\CC_{A^0}^0$ is a $\C$-linear additive braided monoidal category 
with structures induced from $\CC_{A^0}$, 
and the functor $F : \CC_{A^0}^0 \to \CC_{U^0}$; 
$X \mapsto U^0 \boxtimes_{A^0} X$ is a braided tensor functor 
\cite[Theorem 2.67]{CKM2017}, 
where $\CC_{U^0}$ is the category of $U^0$-modules.

We fix $0 \le i_1, i_2 \le k-1$ and $0 \le l_1, l_2 < 2k$. 
Since the category $\CC_{A^0}^0$ is closed under the fusion product, 
we have
\begin{equation}\label{eq:fusion_prod_A0}
  X(i_1,0,l_1) \boxtimes_{A^0} X(i_2,0,l_2) 
  = \sum_{X(i_3,j_3,l_3) \in \Irr^0(A^0)} n(i_3,j_3,l_3) X(i_3,j_3,l_3),
\end{equation}
where 
\begin{equation*}
  n(i_3,j_3,l_3) = \dim I_{A^0} \binom{X(i_3,j_3,l_3)}{X(i_1,0,l_1) \quad X(i_2,0,l_2)}
\end{equation*}
is the fusion rule, that is, the dimension of the space of intertwining operators 
of type $\binom{X(i_3,j_3,l_3)}{X(i_1,0,l_1) \quad X(i_2,0,l_2)}$.
Let
\begin{equation*}
  n(i_3,l_3) = \dim I_{U^0} \binom{U^{i_3,l_3}}{U^{i_1,l_1} \quad U^{i_2,l_2}}
\end{equation*}
be the fusion rule of the irreducible $U^0$-modules $U^{i_p,l_p}$, $p = 1,2,3$.
Then
\begin{equation}\label{eq:fusion_prod_U0-1}
  U^{i_1,l_1} \boxtimes_{U^0} U^{i_2,l_2} 
  = \sum_{U^{i_3,l_3} \in \Irr(U^0)} n(i_3,l_3) U^{i_3,l_3}.
\end{equation}

Since $U^{i,l} = U^0 \boxtimes_{A^0} X(i,j,l)$ for any $0 \le j < k-1$, 
and since the functor $F : \CC_{A^0}^0 \to \CC_{U^0}$; 
$X \mapsto U^0 \boxtimes_{A^0} X$ is a braided tensor functor, 
it follows from \eqref{eq:fusion_prod_A0} that 
\begin{equation*}
  U^{i_1,l_1} \boxtimes_{U^0} U^{i_2,l_2} 
  = \sum_{X(i_3,j_3,l_3) \in \Irr^0(A^0)} n(i_3,j_3,l_3) U^{i_3,l_3}.
\end{equation*}
Thus $n(i_3,l_3) = \sum_{j_3 = 0}^{k-2} n(i_3,j_3,l_3)$ 
by \eqref{eq:fusion_prod_U0-1}.

Now, 
\begin{align*}
  n(i_3,j_3,l_3) 
  &= 
  \dim I_{M_{(k-1)}^0} \binom{ M_{(k-1)}^{i_3,j_3} }{ M_{(k-1)}^{i_1,0} \quad M_{(k-1)}^{i_2,0} }\\
  &\qquad \cdot 
  \dim I_{V_{\Z d}} 
  \binom{ V_{\Z d - l_3 d/2k + (i_3-2j_3) d/2(k-1)} }
  { V_{\Z d - l_1 d/2k + i_1 d/2(k-1)} 
  \quad V_{\Z d - l_2 d/2k + i_2 d/2(k-1)} }
\end{align*}
by \cite[Theorem 2.10]{ADL2005}. 
The first term of the right hand side of the above equation is $0$ or $1$, 
and it is $1$ if and only if 
$i_3 \in R(i_1,i_2)$ and $i_3 - 2j_3 \equiv i_1 + i_2 \pmod{2(k-1)}$ 
by \eqref{eq:paraf_fusion} for $k-1$ in place of $k$.  
The second term of the right hand side is $0$ or $1$, 
and it is $1$ if and only if 
\begin{equation}\label{eq:condition_on_l}
 - (l_1+l_2)(k-1) + (i_1+i_2)k \equiv - l_3(k-1) + (i_3-2j_3)k \pmod{2(k-1)k}
\end{equation}
by the fusion rule for $V_{\Z d}$ \cite[Chapter 12]{DL1993}. 
Hence $n(i_3,j_3,l_3)$ is $0$ or $1$, and  
it is $1$ if and only if $i_3 \in R(i_1,i_2)$, $i_3 - 2j_3 \equiv i_1 + i_2 \pmod{2(k-1)}$, and 
the condition \eqref{eq:condition_on_l} is satisfied. 

If $i_3 \in R(i_1,i_2)$, then there is a unique $0 \le j_3 < k-1$ such that 
$i_3 - 2j_3 \equiv i_1 + i_2 \pmod{2(k-1)}$. 
Moreover, if $i_3 - 2j_3 \equiv i_1 + i_2 \pmod{2(k-1)}$, 
then the condition \eqref{eq:condition_on_l} is 
equivalent to the condition that $l_1 + l_2 \equiv l_3 \pmod{2k}$. 
Therefore, \eqref{eq:fusion_prod_U0} holds. 
\end{proof}

We have the following two corollaries.

\begin{corollary}\label{cor:symmetry_in_fusion_rule}
Let $\zeta = \exp(2\pi\sqrt{-1}/k)$.
Then the map
$\varphi : U^{i,l} \mapsto \zeta^{l} U^{i,l}$ for $0 \le i \le k-1$
and $0 \le l < 2k$
defines an automorphism of the fusion algebra of $U^0$ of order $k$
which is compatible with the isomorphism \eqref{eq:Uil_equivalence}.
\end{corollary}

\begin{corollary}\label{cor:simple_currents}
There are exactly 
$2k$ inequivalent simple current $U^0$-modules, 
which are represented by $U^{0,l}$, $0 \le l < 2k$.
\end{corollary}

Let $U^l = U^{0,l}$. 
Then
\begin{equation}\label{eq:dec_Ul}
\begin{split}
  U^l 
  &= \bigoplus_{j=0}^{k-2} 
  M_{(k-1)}^j \otimes V_{\Z d - l d/2k - j d/(k-1)}\\
  &= U^0 \boxtimes_{A^0} X(0,0,l),
\end{split}
\end{equation}
and the set of equivalence classes of simple current $U^0$-modules is 
\begin{equation}\label{eq:sc_of_U0}
  \SC{U^0} = \{ U^l \mid 0 \le l < 2k \} \quad \text{with} \quad 
  U^l \boxtimes_{U^0} U^{l'} = U^{l+l'}. 
\end{equation}

The conformal weight of $U^l$ satisfies 
$h(U^l) \equiv h(X(0,0,l)) \pmod{\Z}$ 
by \eqref{eq:dec_Ul}. 
Hence we have
\begin{equation}\label{eq:h_Ul}
  h(U^l) \equiv \frac{(k-1) l^2}{4k} \pmod{\Z}.
\end{equation}

Since 
\begin{equation}\label{eq:sc_U0_act}
  U^p \boxtimes_{U^0} U^{i,l} = U^{i,l+p}
\end{equation}
by \eqref{eq:fusion_prod_U0},  
the isomorphism \eqref{eq:Uil_equivalence} implies the next lemma. 

\begin{lemma}\label{lem:stab-U0}
Let $0 \le i \le k-1$ and $0 \le p, l < 2k$.  
Then $U^p \boxtimes_{U^0} U^{i,l} = U^{i,l}$ 
if and only if one of the following conditions holds.

\textup{(i)} $p = 0$.

\textup{(ii)} $k$ is odd, $i = (k-1)/2$, and $p = k$.
\end{lemma}

The fusion rules of  $M^0_{(k-1)}$-modules 
are illustrated as follows.
We set $M^0 = M_{(k-1)}^0$ and $M^{i,j} = M_{(k-1)}^{i,j}$ for simplicity of notation.
The irreducible $M^0$-modules are denoted as 
$M^{i,j}$ by using $0 \le i \le k-1$ and $0 \le j < k-1$. 
There is another description of the irreducible $M^0$-modules. 
Take $0 \le q < 2(k-1)$ such that
$q \equiv i - 2j \pmod{2(k-1)}$. 
Let $\widetilde{M}^{i,q} = M^{i,j}$, which is the multiplicity 
of $V_{\Z \gamma_{k-1} + q \gamma_{k-1}/2(k-1)}$ 
in the decomposition \eqref{eq:Lsl2ki_dec} of $L_{\whmfsl_2}(k-1,i)$. 
Then the fusion product \eqref{eq:paraf_fusion} for $k-1$ in place of $k$ 
can be written as 
\begin{equation}\label{eq:paraf_fusion-new}
  \widetilde{M}^{i_1,q_1} \boxtimes_{M^0} \widetilde{M}^{i_2,q_2} 
  = \sum_{r \in R(i_1,i_2)} \widetilde{M}^{r,q_1+q_2}.
\end{equation}

The relationship between \eqref{eq:fusion_prod_U0} and \eqref{eq:paraf_fusion-new} 
is clear. 
Moreover, the isomorphisms $M^{i,j} \cong M^{k-1-i,j-i}$ 
and $M^{i,j} \circ \theta \cong M^{i,i-j}$ can be written as 
\[
  \widetilde{M}^{i,q} \cong \widetilde{M}^{k-1-i,k-1+q}, 
  \quad \widetilde{M}^{i,q} \circ \theta \cong \widetilde{M}^{i,-q}.
\]

It is known that 
a map defined by $\widetilde{M}^{i,q} \mapsto \eta^q \widetilde{M}^{i,q}$ 
with $\eta = \exp(2\pi\sqrt{-1}/(k-1))$ is compatible with the isomorphism 
$\widetilde{M}^{i,q} \cong \widetilde{M}^{k-1-i,k-1+q}$,  
and it induces an automorphism of the fusion algebra of $M^0$ of order $k-1$.

\section{Vertex operator (super)algebra $U_D$}\label{sec:UD}

In this section, we introduce a positive definite integral lattice $\Gamma_D$ and 
a vertex operator algebra or a vertex operator superalgebra 
$U_D$ for a $\Z_{2k}$-code $D$.

\subsection{Irreducible $U^0$-modules in $V_{N^\circ}$}\label{subsec:irred_mod_in_V_N-mod}

Let $L^{(k)}$ be the lattice as in \eqref{eq:lattice_L}.
We set
\[
  N=\{ \alpha \in L^{(k)} \mid \langle \alpha, \gamma_k \rangle = 0 \}, 
\]
where $\gamma_k\in L^{(k)}$ is as in \eqref{eq:gammas}.
We denote the dual lattice of $N$ by $N^\circ$.
The lattice $N$ is also considered in Section 4 of \cite{AYY2019}, where 
$L^{(k)}$ and $\gamma_k$ are denoted by $L$ and $\gamma$, respectively.
We show how the irreducible $U^0$-modules $U^{i,l}$ appear in 
the $V_N$-module $V_{N^\circ}$. 

For $0 \le j < k$ and $\bsa = \bsa(k) = (a_1,\ldots,a_k) \in \{0,1\}^k$, 
let 
\begin{equation*}
  N(j,\bsa(k)) = N + \delta_{\bsa(k)} - j \alpha_k + \frac{2j -  b_k}{2k} \gamma_k 
\end{equation*}
be a coset of $N$ in $N^\circ$, which is identical with 
$N(j,\bsa)$ of (4.4) in \cite{AYY2019}   
as $\delta_{\bsa(k)} = \frac{1}{2} \sum_{p=1}^k a_p \alpha_p$ and $\wt(\bsa(k)) = b_k$. 
We have
\begin{equation*}
  V_{L^{(k)} + \delta_{\bsa(k)}} = \bigoplus_{j=0}^{k-1} 
  V_{N(j,\bsa(k))} \otimes V_{\Z\gamma_k +(b_k - 2j)\gamma_k/2k} 
\end{equation*}
by (5.2) of \cite{AYY2019}. 
Let $0 \le l < 2k$ be such that $l \equiv b_k - 2j \pmod{2k}$. 
Then $j \equiv (b_k - l)/2 \pmod{k}$, and the multiplicity of 
$V_{\Z\gamma_k + l \gamma_k/2k}$ in $V_{L^{(k)} + \delta_{\bsa(k)}}$ is 
$V_{N((b_k - l)/2,\bsa(k))}$. 
Therefore, 
\begin{equation*}
  U^{i,l} = \{ v \in V_{N((b_k - l)/2,\bsa(k))} \mid  
  \omega_{(1)}^s v = h^{(s)}_{i_{s}+1, i_{s+1}+1} v, 1 \le s \le k-2 \}
\end{equation*}
with $i = i_{k-1}$ by \eqref{eq:VLa_dec-1} and \eqref{eq:def_Uil-bis}.
In the case where $a_1 = \cdots = a_{k-2} = 0$, 
we have $b_{k-1} = a_{k-1}$ and $b_k = a_{k-1} + a_k$. 
Thus the following lemma holds. 

\begin{lemma}\label{lem:Uil_in_VNia}
Let $0 \le i \le k-1$ and $0 \le l < 2k$.

\textup{(1)} 
Define $a_{k-1}$, $a_k \in \{ 0,1 \}$, and $0 \le j < k$ by the conditions
\begin{equation}\label{eq:cond_on_abj}
  i \equiv a_{k-1}, \ l \equiv a_{k-1}+a_k \pmod{2}, 
  \quad j \equiv (a_{k-1} + a_k - l)/2 \pmod{k}.
\end{equation}
Then
\begin{equation*}
  U^{i,l} = \{ v \in V_{N(j, (0,\ldots,0,a_{k-1},a_k))} 
  \mid \omega^s_{(1)} v = 0, 1 \le s \le k-3, \omega^{k-2}_{(1)} v = h^{(k-2)}_{1,i+1} v \}.
\end{equation*}

\textup{(2)} 
In the case $i = 0$, 
we have $a_{k-1} = 0$, $l \equiv a_k \pmod{2}$, and $j \equiv(a_k - l)/2 \pmod{k}$. 
In particular,
\begin{equation*}
  U^l = \{ v \in V_{N(j, (0,\ldots,0,a_k))} \mid \omega^s_{(1)} v = 0, 1 \le s \le k-2 \}.
\end{equation*}
\end{lemma}

In the assertion (2) of the above lemma, we have 
$N(j, (0,\ldots,0,a_k)) = N - l d/2k$,  
as $d = \gamma_k - k \alpha_k$.
Thus 
\begin{equation}\label{eq:Ul_in_Nl}
  U^l = \{ v \in V_{N - l d/2k} \mid \omega^s_{(1)} v = 0, 1 \le s \le k-2 \}.
\end{equation}

\subsection{$\Gamma_D$ and $U_D$}\label{subsec:gamma_D-U_D}

The arguments in this subsection are parallel to those in Section 7 of \cite{AYY2019}. 
For simplicity of notation, set 
\[
  \widetilde{N}^{(l)} = N - l d/2k, \quad 0 \le l < 2k.
\]
We can regard the index $l$ of $\widetilde{N}^{(l)}$ as $l \in \Z_{2k}$. 
Since $\la x, y \ra \in 2\Z$ and $\la x, d/2k \ra \in \Z$ for $x, y \in N$, 
and since $\la d, d \ra = 2(k-1)k$, the following lemma holds. 

\begin{lemma}\label{lem:inner_prod_Np_Nq}
Let $0 \le p, q < 2k$.

  \textup{(1)} 
  $\la \alpha, \beta \ra \in \frac{k-1}{2k} pq + \Z$ 
  for $\alpha \in \widetilde{N}^{(p)}$ and $\beta \in \widetilde{N}^{(q)}$. 

  \textup{(2)} 
  $\la \alpha, \alpha \ra \in \frac{k-1}{2k} p^2 + 2\Z$ 
  for $\alpha \in \widetilde{N}^{(p)}$. 
\end{lemma}

We fix a positive integer $\ell$. 
Define a coset $\widetilde{N}(\xi)$ of $N^\ell$ in $(N^\circ)^\ell$ by 
\begin{equation*}
  \widetilde{N}(\xi) 
  = \{ (x_1,\ldots,x_\ell) \mid x_r \in \widetilde{N}^{(\xi_r)}, 1 \le r \le \ell \} 
\end{equation*}
for $\xi = (\xi_1, \ldots, \xi_\ell) \in (\Z_{2k})^\ell$. 
Then $\widetilde{N}(\xi) + \widetilde{N}(\eta) = \widetilde{N}(\xi + \eta)$ 
for $\xi, \eta \in (\Z_{2k})^\ell$. 

For $\xi = (\xi_1, \ldots, \xi_\ell),\, \eta = (\eta_1,\ldots, \eta_\ell) \in \Z^\ell$, 
define an integer $\xi \cdot \eta$ by 
\[ 
  \xi \cdot \eta = \xi_1 \eta_1 + \cdots + \xi_\ell \eta_\ell, 
\]
and consider a $\Z$-bilinear map 
\begin{equation*}
  \Z^\ell \times \Z^\ell \to \Q/\Z; 
  \quad (\xi, \eta) \mapsto \frac{k-1}{2k} \xi \cdot \eta + \Z.
\end{equation*}

Since $\frac{k-1}{2k} \xi_r \eta_r + \Z$ depends only on $\xi_r$ and $\eta_r$ modulo $2k$, 
the above $\Z$-bilinear map induces a $\Z$-bilinear map 
\[
  (\Z_{2k})^\ell \times (\Z_{2k})^\ell \to \Q/\Z; 
  \quad (\xi, \eta) \mapsto \frac{k-1}{2k} \xi \cdot \eta + \Z, 
\]
where $\xi \cdot \eta = \xi_1 \eta_1 + \cdots + \xi_\ell \eta_\ell$ is considered for 
integers $\xi_r$, $\eta_r$ such that $0 \le \xi_r, \eta_r < 2k$, $1 \le r \le \ell$. 
The $\Z$-bilinear map is non-degenerate if $k$ is even, whereas it is degenerate 
with radical $\{ 0, k \}^\ell$ if $k$ is odd. 
The following lemma holds by Lemma \ref{lem:inner_prod_Np_Nq}.

\begin{lemma}\label{lem:inner_prod_Nxi_Neta}
Let $\xi, \eta \in (\Z_{2k})^\ell$.

  \textup{(1)} 
  $\la \alpha, \beta \ra \in \frac{k-1}{2k} \xi \cdot \eta + \Z$ 
  for $\alpha \in \widetilde{N}(\xi)$ and $\beta \in \widetilde{N}(\eta)$. 

  \textup{(2)} 
  $\la \alpha, \alpha \ra \in \frac{k-1}{2k} \xi \cdot \xi + 2\Z$ 
  for $\alpha \in \widetilde{N}(\xi)$. 
\end{lemma}

Let $(\xi | \eta) = \xi_1 \eta_1 + \cdots + \xi_\ell \eta_\ell \in \Z_{2k}$ 
be the standard inner product of  
$\xi = (\xi_1, \ldots, \xi_\ell)$ and $\eta = (\eta_1,\ldots, \eta_\ell) \in (\Z_{2k})^\ell$. 
Then $\frac{1}{2k} (\xi | \eta) + \Z$ makes sense, 
and
\begin{equation*}
  \frac{1}{2k} (\xi | \eta) + \Z = \frac{1}{2k} \xi \cdot \eta + \Z.
\end{equation*}
The inner product $(\xi | \eta)$ will be used in Section \ref{sec:rep_of_U_D}.

\begin{remark}
The Euclidean weight $\wt_{\mathrm E} (\xi)$ of 
$\xi = (\xi_1, \ldots, \xi_\ell) \in (\Z_{2k})^\ell$ 
is defined as 
\[
  \wt_{\mathrm E} (\xi) = \sum_{r=1}^\ell \min \{ \xi_r^2, (2k - \xi_r)^2 \} \in \Z,
\]
where $\xi_r$ are considered to be integers such that $0 \le \xi_r < 2k$, 
$1 \le r \le \ell$. 
Note that $\frac{1}{2k} \wt_{\mathrm E} (\xi) + 2\Z = \frac{1}{2k} \xi \cdot \xi + 2\Z$. 
The Euclidean weight was used in \cite{KLY2001}. 
\end{remark}

Let $D$ be a $\Z_{2k}$-code of length $\ell$, that is, 
$D$ is an additive subgroup of $(\Z_{2k})^\ell$. 
Set
\begin{equation*}
  \Gamma_D = \bigcup_{\xi \in D} \widetilde{N}(\xi),  
\end{equation*}
which is a sublattice of $(N^\circ)^{\ell}$.  
We consider two cases, namely,  

\medskip\noindent
{\bfseries Case A.} \  
$\frac{k-1}{2k} \xi \cdot \xi \in 2\Z$ for all $\xi \in D$.

\medskip\noindent
{\bfseries Case B.} \ 
$\frac{k-1}{2k} \xi \cdot \eta \in \Z$ for all $\xi, \eta \in D$, and 
$\frac{k-1}{2k} \xi \cdot \xi \in 2\Z + 1$ for some $\xi \in D$.

\medskip
Lemma \ref{lem:inner_prod_Nxi_Neta} implies the following lemma. 

\begin{lemma}\label{lem:GD}
\textup{(1)} $\Gamma_D$ is a positive definite even lattice if and only if 
$D$ is in Case A.

\textup{(2)} $\Gamma_D$ is a positive definite odd lattice if and only if 
$D$ is in Case B. 
\end{lemma}

In Case A, $V_{\Gamma_D}$ is a vertex operator algebra. 
In Case B, set 
\[
  D^0 = \{ \xi \in D \mid {\textstyle \frac{k-1}{2k} \xi \cdot \xi \in 2\Z} \}, \quad
  D^1 = \{ \xi \in D \mid {\textstyle \frac{k-1}{2k} \xi \cdot \xi \in 2\Z + 1} \}.
\]
Then $D^0$ is a subgroup of $D$, and $D = D^0 \cup D^1$ is 
the coset decomposition of $D$ by $D^0$. 
Let $\Gamma_{D^p} =  \bigcup_{\xi \in D^p} \widetilde{N}(\xi)$, $p = 0,1$.
Then 
$V_{\Gamma_D} = V_{\Gamma_{D^0}} \oplus V_{\Gamma_{D^1}}$ 
is a vertex operator superalgebra.  

We have 
$V_{\widetilde{N}(\xi)} 
= V_{\widetilde{N}^{(\xi_1)}} \otimes \cdots \otimes V_{\widetilde{N}^{(\xi_\ell)}} 
\subset (V_{N^\circ})^{\otimes \ell}$, and 
$V_{\Gamma_D} = \bigoplus_{\xi \in D} V_{\widetilde{N}(\xi)}$. 
Let 
\begin{equation*}
  U_\xi = \{ v \in V_{\widetilde{N}(\xi)} \mid (\om_{S^{\otimes \ell}})_{(1)} v = 0 \},
\end{equation*}
where 
$\om_{S^{\otimes \ell}}$ is the conformal vector of the vertex operator subalgebra 
$S^{\otimes \ell}$ of $(V_N)^{\otimes \ell}$ with 
$S = L(c_1,0) \otimes \cdots \otimes L(c_{k-2}, 0)$.  
Then 
$U_\xi = U^{\xi_1} \otimes \cdots \otimes U^{\xi_\ell}$ 
by \eqref{eq:Ul_in_Nl}. 
In particular, 
$U_\0 = (U^0)^{\otimes \ell}$ for the zero codeword $\0 = (0,\ldots,0)$. 
We see from \eqref{eq:sc_of_U0} that 
the set of equivalence classes of simple current $U_\0$-modules is 
\begin{equation*}
  \SC{U_\0} = \{ U_\xi \mid \xi \in (\Z_{2k})^\ell \} \quad \text{with} \quad 
  U_\xi \boxtimes_{U_\0} U_{\xi'} = U_{\xi+\xi'}. 
\end{equation*}
The conformal weight $h(U_\xi)$ of $U_\xi$ is 
\begin{equation*}
  h(U_\xi) \equiv \frac{k-1}{4k} \xi \cdot \xi \pmod{\Z}
\end{equation*}
by \eqref{eq:h_Ul}. 
Hence $h(U_\xi) \in \Z$ for $\xi \in D$ if $D$ is in Case A.

The next proposition follows from Theorems \ref{thm:U0} and \ref{thm:irr_U0-mod}.

\begin{proposition}\label{prop:U_0}
$U_\0 =  (U^{0})^{\otimes \ell}$ is a simple, self-dual, rational, and $C_2$-cofinite 
vertex operator algebra of CFT-type with central charge $3 \ell (k-1)/(k+1)$. 
Any irreducible $U_\0$-module except for $U_\0$ itself has positive conformal weight.
\end{proposition}

Let $U_D$ be the commutant of $S^{\otimes \ell}$ in $V_{\Gamma_D}$. 
Then
\begin{equation*}
  U_D = \{ v \in V_{\Gamma_D} \mid (\om_{S^{\otimes \ell}})_{(1)} v = 0\} 
  = \bigoplus_{\xi \in D} U_\xi, 
\end{equation*}
so $U_D$ is a $D$-graded simple current extension of $U_\0$. 
We have the next theorem. 
\begin{theorem}\label{thm:U_D}
\textup{(1)} 
If $D$ is in Case A, then $U_D$ is a simple, self-dual, rational, and $C_2$-cofinite vertex 
operator algebra of CFT-type with central charge $3 \ell (k-1)/(k+1)$. 

\textup{(2)} 
If $D$ is in Case B, 
then $U_D = U_{D^0} \oplus U_{D^1}$ is a simple vertex operator superalgebra. 
The even part $U_{D^0}$ and the odd part $U_{D^1}$ are given by 
$U_{D^p} = \bigoplus_{\xi \in D^p} U_\xi$, $p = 0,1$, 
and $h(M_{D^1}) \in \Z + 1/2$.  
\end{theorem}

\section{Representations of $U_D$}\label{sec:rep_of_U_D}

In this section, we construct all the irreducible $\chi$-twisted $U_D$-modules for 
$\chi \in D^*$ in $V_{(N^\circ)^\ell}$, and classify them, where $D^* = \Hom(D, \C^\times)$. 
We argue as in Sections 8 and 9 of \cite{AYY2019}. 

\subsection{Irreducible $U_D$-modules: Case A}\label{subsec:irred_U_D-modules_A}

Let $D$ be a $\Z_{2k}$-code of length $\ell$ 
in Case A of Section \ref{subsec:gamma_D-U_D}. 
Let $b_{U^0} : \SC{U^0} \times \Irr(U^0) \to \Q/\Z$ be a map defined by 
\begin{equation*}
  b_{U^0}(U^p, U^{i,l}) = h(U^p \boxtimes_{U^0} U^{i,l}) - h(U^p) - h(U^{i,l}) + \Z
\end{equation*} 
for $0 \le i \le k-1$ and $0 \le p, l < 2k$. 
Since $h(U^{i,l}) \equiv h(X(i,0,l)) \pmod{\Z}$ 
by \eqref{eq:Uil_as_U0_module},  
we obtain by using \eqref{eq:h_Ul} and \eqref{eq:sc_U0_act} that 
\begin{equation}\label{eq:b_U0}
  b_{U^0}(U^p, U^{i,l}) = p((k-1)l  - ki)/2k + \Z.
\end{equation}

Any irreducible $U_\0$ module is of the form
\begin{equation*}
  U_{\mu,\nu} = U^{\mu_1,\nu_1} \otimes \cdots \otimes U^{\mu_\ell,\nu_\ell}
\end{equation*}
for some $\mu = (\mu_1,\ldots,\mu_\ell)$ with $0 \le \mu_r \le k-1$, $1 \le r \le \ell$, 
and $\nu = (\nu_1,\ldots, \nu_\ell) \in (\Z_{2k})^\ell$.  
We have $U_{\0,\xi} = U_\xi$, and 
\begin{equation*}
  U_\xi \boxtimes_{U_\0} U_{\mu,\nu} = U_{\mu,\nu + \xi}.
\end{equation*}
by \eqref{eq:sc_U0_act}.
Define a map $b_{U_\0} : \SC{U_\0} \times \Irr(U_\0) \to \Q/\Z$ by  
\begin{equation*}
  b_{U_\0}(U_\xi, U_{\mu,\nu}) = 
  h(U_\xi \boxtimes_{U_\0} U_{\mu,\nu}) - h(U_\xi) - h(U_{\mu,\nu}) + \Z.
\end{equation*}
Then it follows from \eqref{eq:b_U0} that 
\begin{equation}\label{eq:b_U_0-2}
  b_{U_\0}(U_\xi,U_{\mu,\nu}) = \frac{1}{2k}(\xi | (k-1)\nu - k \mu) + \Z,
\end{equation}
where $(\, \cdot \,|\, \cdot \,)$ is the standard inner product on $(\Z_{2k})^\ell$. 
Although each entry $\mu_r$ of $\mu$ is an integer such that $0 \le \mu_r \le k-1$, 
we can treat it as an element of $\Z_{2k}$ on the right hand side of \eqref{eq:b_U_0-2}. 
Since $U_\eta = U_{\0,\eta}$ for $\eta \in (\Z_{2k})^\ell$, this in particular implies that
\begin{equation*}
  b_{U_\0}(U_\xi,U_\eta) = \frac{k-1}{2k}(\xi | \eta) + \Z.
\end{equation*}

Let $D^\perp = \{ \eta \in (\Z_{2k})^\ell \mid (D | \eta) = 0\}$.
Then $\abs{D} \abs{D^\perp} = \abs{(\Z_{2k})^\ell}$, as $(\, \cdot \,|\, \cdot \,)$ 
is a non-degenerate bilinear form on $(\Z_{2k})^\ell$. 
Consider a map
\begin{equation*}
  \chi_{U_{\mu, \nu}} : D \to \C^\times; \quad 
  \xi \mapsto \exp(2\pi\sqrt{-1} b_{U_\0}(U_\xi, U_{\mu, \nu})).
\end{equation*}
We have 
\begin{equation}\label{eq:char_of_D-2}
  \chi_{U_{\mu, \nu}}(\xi) = \exp(2\pi\sqrt{-1} (\xi | (k-1)\nu - k \mu)/2k)
\end{equation}
by \eqref{eq:b_U_0-2}. 
Hence $\chi_{U_{\mu, \nu}} \in D^*$. 

\begin{lemma}\label{lem:linear_char_of_D}
\textup{(1)} 
$\chi_{U_{\mu, \nu}} = 1$; the principal character of $D$ if and only if 
$(k-1)\nu - k \mu \in D^\perp$.

\textup{(2)} 
For any $\chi \in D^*$, 
there exists $U_{\mu, \nu} \in \Irr(U_\0)$ such that $\chi = \chi_{U_{\mu, \nu}}$. 
\end{lemma}

\begin{proof}
The assertion (1) is a consequence of  \eqref{eq:char_of_D-2} 
and the definition of $D^\perp$. 
For any $0 \le p < 2k$, we have 
$p \equiv (k-1)l - ki \pmod{2k}$ for some $0 \le i \le k-1$ and $0 \le l < 2k$. 
Hence for any $\eta \in (\Z_{2k})^\ell$, there are 
$\mu = (\mu_1,\ldots,\mu_\ell)$ with $0 \le \mu_r \le k-1$, $1 \le r \le \ell$, 
and $\nu \in (\Z_{2k})^\ell$ such that 
$\eta = (k-1)\nu - k \mu$. 
Since $(\, \cdot \,|\, \cdot \,)$ is non-degenerate on $(\Z_{2k})^\ell$, 
the assertion (2) holds. 
\end{proof}

We consider a coset 
\begin{equation*}
  N(\eta, \delta^{(1)}, \delta^{(2)}) 
  = \{ (x_1, \ldots, x_\ell) \mid x_r \in N(\eta_r, (0, \ldots, 0, d^{(1)}_r, d^{(2)}_r)), 
  1 \le r \le \ell \} 
\end{equation*}
of $N^\ell$ in $(N^\circ)^\ell$ 
for $\eta = (\eta_1, \ldots, \eta_\ell) \in (\Z_k)^\ell$ and 
$\delta^{(s)} = (d^{(s)}_1,\ldots,d^{(s)}_\ell) \in \{0,1\}^\ell$, $s = 1,2$. 
The next proposition holds by Lemma \ref{lem:Uil_in_VNia}.

\begin{proposition}\label{prop:irred_U_0-mod_in_coset}
Let $\mu = (\mu_1,\ldots,\mu_\ell)$ with $0 \le \mu_r \le k-1$, $1 \le r \le \ell$, 
and let $\nu = (\nu_1, \ldots, \nu_\ell) \in (\Z_{2k})^\ell$. 
Define $d^{(1)}_r$, $d^{(2)}_r \in \{ 0,1 \}$, and $0 \le \eta_r < k$ by the 
conditions
\begin{equation*}
  \mu_r \equiv d^{(1)}_r, \  
  \nu_r \equiv d^{(1)}_r + d^{(2)}_r \pmod{2}, \quad 
  \eta_r \equiv (d^{(1)}_r + d^{(2)}_r - \nu_r)/2 \pmod{k}
\end{equation*}
for $1 \le r \le \ell$. 
Then $V_{N(\eta, \delta^{(1)}, \delta^{(2)})}$ contains 
the irreducible $U_\0$-module $U_{\mu, \nu}$. 
\end{proposition}

Let $0 \le i \le k-1$ and $0 \le p, l < 2k$. 
Define $a_{k-1}$, $a_k \in \{ 0,1 \}$, and $0 \le j < k$ by the conditions 
\eqref{eq:cond_on_abj}. 
Then
\[
  \langle \alpha, \beta \rangle  \in p((k-1)l - ki)/2k + \Z
\] 
for $\alpha \in \widetilde{N}^{(p)}$ and $\beta \in N(j, (0,\ldots,0,a_{k-1}, a_k))$.  
Thus the the following lemma holds by \eqref{eq:b_U_0-2}.

\begin{lemma}\label{inner_product_cosets}
Let $\mu$, $\nu$, $\eta$, $\delta^{(1)}$, and $\delta^{(2)}$ be as in 
Proposition \ref{prop:irred_U_0-mod_in_coset}, and let $\xi \in (\Z_{2k})^\ell$. 
Then $\la x, y \ra \in b_{U_\0}(U_\xi,U_{\mu,\nu})$ 
for $x \in \widetilde{N}(\xi)$ and $y \in N(\eta, \delta^{(1)}, \delta^{(2)})$.
\end{lemma}

Let $X \in \Irr(U_\0)$. 
Then $X = U_{\mu,\nu}$ for some $\mu$ and $\nu$. 
Let $\eta$, $\delta^{(1)}$, and $\delta^{(2)}$ be as in 
Proposition \ref{prop:irred_U_0-mod_in_coset}. 
Then $X$ is contained in $V_{N(\eta, \delta^{(1)}, \delta^{(2)})}$. 
Since $U_\xi$ is contained in $V_{\widetilde{N}(\xi)}$, 
and since the cosets $\widetilde{N}(\xi) + N(\eta, \delta^{(1)}, \delta^{(2)})$ 
of $N^\ell$ in $(N^\circ)^\ell$ are distinct for all $\xi \in D$, 
the $\chi_X$-twisted $U_D$-submodule $U_D \cdot X$ 
of $V_{(N^\circ)^\ell}$ generated by $X$ is isomorphic to 
$U_D \boxtimes_{U_\0} X 
= \bigoplus_{\xi \in D} U_\xi \boxtimes_{U_\0} X$. 
If $\chi_X(\xi) = 1$ for all $\xi \in D$, then 
$N(\eta, \delta^{(1)}, \delta^{(2)}) \subset (\Gamma_D)^\circ$ by 
Lemma \ref{inner_product_cosets}, and we have 
$U_D \cdot X \subset V_{(\Gamma_D)^\circ}$. 
Thus the following theorem holds. 

\begin{theorem}\label{thm:contain_irred}
\textup{(1)} 
Any irreducible $\chi$-twisted $U_D$-module, $\chi \in D^*$, 
is contained in $V_{(N^\circ)^\ell}$.

\textup{(2)} 
Any irreducible untwisted $U_D$-module is contained in $V_{(\Gamma_D)^\circ}$.
\end{theorem}

Define an action of $D$ on $\Irr(U_\0)$ by 
$X \mapsto U_\xi \boxtimes_{U_\0} X$
for $\xi \in D$ and $X \in \Irr(U_\0)$. 
Let $\Irr(U_\0) = \bigcup_{i\in I} \orbit_i$ 
be the $D$-orbit decomposition, and let 
$D_X = \{ \xi \in D \mid U_\xi \boxtimes_{U_\0} X = X \}$ 
be the stabilizer of $X$. 
The next lemma holds by Lemma \ref{lem:stab-U0}. 

\begin{lemma}\label{lem:exceptional_id_U0}
$U_\xi \boxtimes_{U_{\0}} U_{\mu,\nu} = U_{\mu,\nu}$  
for some $\xi \ne \0$ if and only if 
$k$ is odd, $\xi = (\xi_1,\ldots,\xi_\ell) \in \{0, k\}^\ell$, and 
$\mu_r = (k-1)/2$ for $1 \le r \le \ell$ such that $\xi_r = k$. 
\end{lemma}

We study the structure of $U_D \boxtimes_{U_\0} X$ for $X \in \Irr(U_\0)$. 
If $D_X = 0$, 
then $U_D \boxtimes_{U_\0} X$ is an irreducible $\chi_X$-twisted $U_D$-module.  

Suppose $D_X \ne 0$. 
Then $k$ is odd, and $D_X \subset \{0, k\}^\ell$ by Lemma \ref{lem:exceptional_id_U0}. 
Let $C = \{ (0), (k) \}$ be a $\Z_{2k}$-code of length one consisting of two codewords 
$(0)$ and $(k)$. 
The code $C$ is in Case A or in Case B 
according as $k \equiv 1$ or $k \equiv 3 \pmod{4}$. 
Hence the $\Z_2$-graded simple current extension $U_C = U^0 \oplus U^k$ of $U^0$ 
is a simple vertex operator algebra with $h(U^k) \in \Z$ or 
a simple vertex operator superalgebra with $h(U^k) \in \Z + 1/2$ 
according as $k \equiv 1$ or $k \equiv 3 \pmod{4}$. 
We can regard any additive subgroup of $\{ 0, k\}^\ell \subset (\Z_{2k})^\ell$ 
as an additive subgroup of $(\Z_2)^\ell$ 
under the correspondence $0 \mapsto 0$ and $k \mapsto 1$. 
Since $k$ is odd, the correspondence is the reduction modulo $2$, 
and it gives an isometry from 
$(\{ 0, k\}^\ell, (\, \cdot\, | \, \cdot \,))$ to $((\Z_2)^\ell, (\, \cdot\, | \, \cdot \,))$, 
where $(\, \cdot\, | \, \cdot \,)$ is 
the standard inner product on either $(\Z_{2k})^\ell$ or $(\Z_2)^\ell$. 
In particular, $D_X \cap D_X^\perp$ in $(\Z_{2k})^\ell$ corresponds to 
$D_X \cap D_X^\perp$ in $(\Z_2)^\ell$. 
Thus the following theorem holds by Propositions 2.3, 2.5, and 2.6 of \cite{AYY2019}.  

\begin{theorem}\label{thm:irred_U_D-mod}
Let $X \in \Irr(U_\0)$. 

\textup{(1)} 
If $D_X = 0$, 
then $U_D \boxtimes_{U_\0} X$ is an irreducible $\chi_X$-twisted $U_D$-module.

\textup{(2)} 
Suppose $k$ is odd and $D_X \ne 0$.

If $k \equiv 1 \pmod{4}$, then 
$U_D \boxtimes_{U_\0} X = \bigoplus_{j = 1}^{\abs{D_X}} V^j$, 
where $V^j$, $1 \le j \le \abs{D_X}$, are inequivalent irreducible 
$\chi_X$-twisted $U_D$-modules. 
Furthermore, 
$V^j \cong \bigoplus_{W \in \orbit_i} W$ 
as $U_\0$-modules, where $\orbit_i$ is the $D$-orbit in $\Irr(U_\0)$ 
containing $X$.

If $k \equiv 3 \pmod{4}$, then 
$U_D \boxtimes_{U_\0} X = \bigoplus_{j=1}^{\abs{D_X \cap D_X^\perp}} (V^j)^{\oplus m}$, 
where $m = [D_X : D_X \cap D_X^\perp]^{1/2}$, and 
$V^j$, $1 \le j \le \abs{D_X \cap D_X^\perp}$, are inequivalent irreducible 
$\chi_X$-twisted $U_D$-modules. 
Furthermore, 
$V^j \cong \bigoplus_{W \in \orbit_i} W^{\oplus m}$ 
as $U_\0$-modules, where $\orbit_i$ is the $D$-orbit in $\Irr(U_\0)$ 
containing $X$.
\end{theorem}

Any irreducible $\chi$-twisted $U_D$-module, $\chi \in D^*$,   
is isomorphic to a direct summand 
of $U_D\boxtimes_{U_{\0}} X$ with $\chi = \chi_X$ for some $X \in \Irr(U_\0)$. 
Thus the classification of irreducible $\chi$-twisted $U_D$-modules 
for any $\chi \in D^*$ is obtained by  
Theorem \ref{thm:irred_U_D-mod}. 

We can write $\chi_i$ for $\chi_X$, and $D_i$ for $D_X$ 
if $X$ belongs to a $D$-orbit $\orbit_i$ in $\Irr(U_\0)$, 
as $\chi_X$ and $D_X$ are independent of the choice of $X \in \orbit_i$.
Let $I(\chi) = \{ i \in I \mid \chi_i = \chi \}$, 
which is non-empty by Lemma \ref{lem:linear_char_of_D}. 
%

By the above arguments, we obtain the next theorem.

\begin{theorem}\label{thm:count_irred_twisted_mod}
The number of inequivalent irreducible 
$\chi$-twisted $U_D$-modules for $\chi \in D^*$ is as follows.
\begin{alignat*}{2}
& \abs{I(\chi)} & \quad 
& \text{if $k$ is even},\\
& \abs{I(\chi)_0} + \sum_{i \in I(\chi)_1} \abs{D_i} & \quad 
& \text{if } k \equiv 1 \pmod{4},\\
& \abs{I(\chi)_0} + \sum_{i \in I(\chi)_1} \abs{D_i \cap D_i^\perp} & \quad 
& \text{if } k \equiv 3 \pmod{4}, 
\end{alignat*}
where $I(\chi)_0 = \{ i \in I(\chi) \mid D_i = 0 \}$ and 
$I(\chi)_1 = I(\chi) \setminus I(\chi)_0$. 
\end{theorem}

\subsection{Irreducible $U_D$-modules: Case B}\label{subsec:irred_U_D-modules_B}

Let $D$ be a $\Z_{2k}$-code of length $\ell$ in Case B of Section \ref{subsec:gamma_D-U_D},  
and let $D^0$ and $D^1$ be as in Section \ref{subsec:gamma_D-U_D}. 
Since $D^0$ is a $\Z_{2k}$-code of length $\ell$ in Case A, 
we see from Section \ref{subsec:irred_U_D-modules_A} that any irreducible 
$U_{D^0}$-module $P$ is isomorphic to a direct summand of 
$U_{D^0} \boxtimes_{U_\0} X$ for some $X \in \Irr(U_\0)$, 
and that $X$ is contained in $V_{N(\eta, \delta^{(1)}, \delta^{(2)})}$ 
for some coset $N(\eta, \delta^{(1)}, \delta^{(2)})$ of $N^\ell$ in $(\Gamma_{D^0})^\circ$. 
Since $U_D = U_{D^0} \oplus U_{D^1}$, 
the $U_D$-submodule $U_D \cdot P$ of $V_{(\Gamma_{D^0})^\circ}$ 
generated by $P$ is isomorphic to 
$U_D \boxtimes_{U_{D^0}} P$. 
Moreover, $U_D \boxtimes_{U_{D^0}} P$ is 
either an irreducible $U_D$-module or a direct sum of two irreducible $U_D$-modules.  
Since any irreducible $U_D$-module is obtained in this way,  
the following theorem holds.

\begin{theorem}\label{thm:contain_irred_B}
  Any irreducible $U_D$-module is contained in $V_{(\Gamma_{D^0})^\circ}$.
\end{theorem}

\appendix

\section{Top level of $U^l$, $0 \le l < 2k$}

In this appendix,  we prove the following theorem on the top level of 
$U^l$, $0 \le l < 2k$, defined in \eqref{eq:dec_Ul}.

\begin{theorem}\label{thm:top_level_Ul}
The weight and the dimension of the top level of the simple current 
$U^0$-module $U^l$, $0 \le l < 2k$, are as follows.

\textup{(1)} 
If $l = 0$, then the weight is $0$ and the dimension is $1$. 

\textup{(2)}
If $l$ is odd, 
then the weight is $l(2k - l)/4k - 1/4$ and the dimension is $1$. 

\textup{(3)} 
If $l \ne 0$ is even, 
then the weight is $l(2k - l)/4k$ and the dimension is $2$. 
\end{theorem}

\begin{proof}
Since $U^l \circ \theta \cong U^{-l} = U^{2k-l}$ by \eqref{eq:theta_act_Uil}, 
it is enough to consider the case $0 \le l \le k$. 
The top level of $U^0$ is $\C\1$, and the assertion (1) holds. 
Thus we assume that $1 \le l \le k$. 
If $k = 2$, then $U^l = V_{\Z d - l d/4}$ with $\langle d,d \rangle = 4$, 
as $M_{(1)}^0 = \C\1$. 
Hence the theorem holds for $k = 2$. 
So we assume that $k \ge 3$. 

Recall the notation $X(i,j,l)$ in Section \ref{sec:VOA_U0}. 
For a fixed $l$, let $P(j)$ be the conformal weight of $X(0,j,l)$. 
Since $U^l = \bigoplus_{j=0}^{k-2} X(0,j,l)$,  
we need to calculate the minimum value of $P(j)$ for integers $j$ in the range 
$0 \le j < k-1$. 
We have
\begin{equation}\label{eq:Pj_range1}
\begin{split}
  P(j) 
  &= \frac{j(k-1-j)}{k-1} + \frac{( (k-1)l + 2kj )^2}{4(k-1)k}\\
  &= \big( j + \frac{l+1}{2} \big)^2 - \frac{(l+1)^2}{4} + \frac{(k-1) l^2}{4k}
\end{split}
\end{equation}
for $j$ in the range 
\begin{equation}\label{eq:range1_j}
  0 \le j \le (k-1)(k-l)/2k, 
\end{equation}
and 
\begin{equation}\label{eq:Pj_range2}
\begin{split}
  P(j) 
  &= \frac{j(k-1-j)}{k-1} + \frac{( (k-1)l + 2kj - 2(k-1)k )^2}{4(k-1)k}\\
  &= \big( j - ( k - \frac{l+1}{2} ) \big)^2 + \frac{l(2k - l)}{4k} - \frac{1}{4}
\end{split}
\end{equation}
for $j$ in the range 
\begin{equation}\label{eq:range2_j}
  (k-1)(k-l)/2k \le j < k-1. 
\end{equation}

The dimension of the top level of $V_{\Z d - l d/2k - j d/(k-1)}$ 
is $2$ if $(k-1)l + 2kj = (k-1)k$, otherwise it is $1$. 
Since the dimension of the top level of $M_{(k-1)}^j$ is $1$, 
the dimension of the top level of $X(0,j,l)$ is 
$2$ if $l = k$ and $j = 0$, otherwise it is $1$, 
as $1 \le l \le k$. 

The minimum value of the quadratic polynomial $P(j)$ for integers $j$ in the range 
\eqref{eq:range1_j} is $(k-1)l^2/4k$ at $j = 0$ by \eqref{eq:Pj_range1}.
As for the minimum value of $P(j)$ for integers $j$ in the range \eqref{eq:range2_j}, 
note that
\[
(k-1)( k-l )/2k \le k - 1 - (l+1)/2 < k-1,
\]
as $k \ge 3$ and $1 \le l \le k$. 
We argue the cases $l = 1$, $2$, and $l \ge 3$ separately. 

First, assume that $l = 1$. Then $k - (l+1)/2 = k-1$ is not in  
the range \eqref{eq:range2_j}. 
So the minimum value of $P(j)$ for integers $j$ in the range \eqref{eq:range2_j} 
is $5/4 - 1/4k$ at $j = k-2$ by \eqref{eq:Pj_range2}. 
Since $P(0) < P(k-2)$, the assertion (2) holds for $l = 1$. 

Next, assume that $l = 2$. Then $k - (l+1)/2 = k-3/2$, 
so the minimum value of $P(j)$ for integers $j$ in the range \eqref{eq:range2_j} 
is $(k-1)/k$ at $j = k-2$  by \eqref{eq:Pj_range2}. 
Since $P(0) = P(k-2)$, the assertion (3) holds for $l = 2$. 

Now, assume that $3 \le l \le k$. 
Suppose $l$ is odd. 
Then the minimum value of $P(j)$ for integers $j$ in the range \eqref{eq:range2_j} is 
$l(2k - l)/4k - 1/4$ at $j = k-(l+1)/2$ by \eqref{eq:Pj_range2}. 
The minimum value is smaller than $P(0)$. 
Thus the assertion (2) holds. 

Finally, suppose $4 \le l \le k$ and $l$ is even. 
Then the minimum value of $P(j)$ for integers $j$ in the range \eqref{eq:range2_j} is 
$l(2k - l)/4k$ at $j =  k-1-l/2$ and $k-l/2$ by \eqref{eq:Pj_range2}. 
The minimum value is smaller than $P(0)$. 
Thus the assertion (3) holds. 
The proof is complete.
\end{proof}

\end{document}